\RequirePackage{ifpdf}
\ifpdf 
\documentclass[pdftex]{sigma}
\else
\documentclass{sigma}
\fi

\def\cal{\mathcal}

\newcommand{\nn}[1]{(\ref{#1})}

\newcommand{\sfrac}[2]{{\textstyle \frac{#1}{#2}}}

\newcommand{\Proj}{\operatorname{Proj}}
\newcommand{\Ric}{\operatorname{Ric}}
\newcommand{\si}{\sigma}
\newcommand{\bg}{\mbox{\boldmath{$ g$}}}
\newcommand{\g}{g_+}
\def\la{\lambda}

\def\tlongrightarrow{\relbar\joinrel\relbar\joinrel
                        \relbar\joinrel\relbar\joinrel\longrightarrow}

\def\pkm{P_{k,{\bf m},m_j\,,k-1-m_j}}

\begin{document}

\allowdisplaybreaks
	
\renewcommand{\PaperNumber}{100}

\FirstPageHeading

\renewcommand{\thefootnote}{$\star$}

\ShortArticleName{Conformal Dirichlet--Neumann Maps and Poincar\'e--Einstein Manifolds}

\ArticleName{Conformal Dirichlet--Neumann Maps\\ and Poincar\'e--Einstein Manifolds\footnote{This paper is a
contribution to the Proceedings of the 2007 Midwest
Geometry Conference in honor of Thomas~P.\ Branson. The full collection is available at
\href{http://www.emis.de/journals/SIGMA/MGC2007.html}{http://www.emis.de/journals/SIGMA/MGC2007.html}}}

\Author{A.\ Rod GOVER}

\AuthorNameForHeading{A.R. Gover}

\Address{Department of Mathematics,
  The University of Auckland,\\
  Private Bag 92019,
  Auckland 1,
  New Zealand}
\Email{\href{mailto:gover@math.auckland.ac.nz}{gover@math.auckland.ac.nz}}

\URLaddress{\url{http://www.math.auckland.ac.nz/~gover/}}

\ArticleDates{Received October 07, 2007; Published online October 21, 2007}

\Abstract{A conformal description of Poincar\'e--Einstein manifolds is
  developed: these structures are seen to be a special case of a
  natural weakening of the Einstein condition termed an almost
  Einstein structure. This is used for two purposes: to shed light on
  the relationship between the scattering construction of
  Graham--Zworski and the higher order conformal Dirichlet--Neumann maps
  of Branson and the author; to sketch a new construction of non-local
  (Dirichlet--to--Neumann type) conformal operators between tensor bundles.}


\Classification{58J40; 53A30; 58J32}


\section{Introduction}

A Poincar\'e--Einstein manifold is a compact manifold with boundary
equipped with a negative scalar curvature Einstein metric that is
``conformally compact''; that is, it has a certain conformal scale
singularity at the boundary. The model that this generalises is the
dimension $n+1$ hyperbolic ball with the sphere $\mathbb{S}^n$ as the
conformal boundary at inf\/inity. As in the model case, the
Poincar\'e--Einstein boundary has a conformal structure and a central
theme in the study of these is to relate this to the Riemannian
geometry of the interior. Much of the motivation has come from
physics, in particular in connection with the so-called AdS/CFT
(Anti-de Sitter/Conformal Field Theory) correspondence suggested by
Maldacena \cite{Mal,Witten}. A number of purely geometric questions
arise and these have generated signif\/icant mathematical interest, see
for example \cite{Albin,And,CQYrenromvol}.

In the seminal work \cite{GrZ} Graham and Zworski developed the
scattering theory for (asymptoti\-cally) Poincar\'e--Einstein manifolds
and its use for describing conformal objects on the boundary. This
exploited the available picture for the scattering theory of inf\/inite
volume hyperbolic quotients (see e.g.~\cite{Perry} and references
therein), and for asymptotically hyperbolic manifolds (e.g.~\cite{JoSB}) as well as the results for the spectrum of the Laplacian on
these due to Mazzeo and Mazzeo--Melrose \cite{MazzeoH,MazzeoU,MazMel}.

A focus of \cite{GrZ} was to extract and study, via the scattering
machinery, the conformal Laplacian type operators of \cite{GJMS} as
well as Branson's $Q$-curvature \cite{Tomsharp}. However the scattering
operator they construct is an essentially global object and, as they
point out, may be viewed as a generalised Dirichlet--to--Neumann map.
This picture of the operator is especially relevant for certain real
values of the spectral parameter and this is a point we wish to
underscore here. In particular one aim here is to shed light on the
relationship between the Graham--Zworski construction and the conformal
Dirichlet--to--Neumann maps constructed by Branson with the author
\cite{BrGonon}.  In the spirit of the standard Dirichlet--to--Neumann
construction the latter are based around elliptic source problems and
target problems; the interior operators are conformal Laplacian type
operators of the form $\Delta^k+ lower~order~terms$, developed with
Eastwood \cite{GoSrni99}, and the the boundary operators are derived
from the conformal tractor calculus of hypersurface type submanifolds.
Both the scattering operator and the Dirichlet--to--Neumann maps of
\cite{BrGonon} are conformally invariant. In the homogeneous setting
this translates to the fact that they intertwine principal series
representations and so, by naturality and the uniqueness of the
intertwinors concerned \cite{specgen}, it would seem that in this
setting they should agree.  Here we make steps toward an explicit
matching of the constructions.
 As background, the basic conformal machinery is revised in
the next section and then, in Section \ref{DNBG}, we outline the
construction of the conformal Dirichlet--to--Neumann operators from
\cite{BrGonon}.

Let $M$ be a compact smooth manifold with boundary $\Sigma=\partial
M$. A metric $\g$ on the interior~$M_+$ of~$M$ is said to be
conformally compact if it extends (with some specif\/ied regularity) to~$M$ by $g=x^2\g$ where $g$ is non-degenerate up to the boundary, and
$x$ is a non-negative def\/ining function for the boundary (i.e.\
$\Sigma$ is the zero set for $x$, and $d x$ is non-vanishing on
$\Sigma$).  In this situation the metric $\g$ is complete and the
restriction of $g$ to $T\Sigma $ in $TM|_\Sigma$ determines a~conformal structure that is independent of the choice of def\/ining
function $x$; the latter is termed the conformal inf\/inity of $M_+$.
If the def\/ining function is chosen so that $|dx|=1$ (with respect to~$g$) along $\Sigma$ then the sectional curvatures tend to $-1$ at
inf\/inity \cite{MazzeoH} and the structure is said to be asymptotically
hyperbolic.  The manifold is said to be a Poincar\'e--Einstein
structure if, in addition, the interior metric $\g$ is Einstein, with
$\Ric (\g)=-n \g$.  For simplicity of presentation, we shall treat
all structures as smooth. It is straightforward to adapt the results
to the case of limited regularity.  We shall also conf\/ine our
discussion to metrics and conformal structures of positive def\/inite
signature.

Since the ``boundary singularity'' of the metric $\g$ is conformal in
nature, it is natural to explore the meaning of the
Poincar\'e--Einstein structure in terms of the conformal geometry of
$M$. An elegant picture emerges, and this is the subject of Section
\ref{PEgen}.  The equation controlling whether or not a metric is
conformally Einstein is a second order overdetermined partial
dif\/ferential equation, see expression \nn{prim}. A solution of this
determines an Einstein metric only if it is non-vanishing. In fact any
solution is non-vanishing on an open dense set \cite{GoNur} and so we
say that, when equipped with such a solution, a conformal, or
(pseudo-)Riemannian, manifold is almost Einstein \cite{Goalmost}. This
notion is given a geometric interpretation via the conformal tractor
calculus. The conformal tractor connection \cite{Thomas} is a
canonical and conformally invariant vector bundle connection; it is
equivalent \cite{CapGoTAMS} to the conformal Cartan connection of
\cite{Cart}.  From the development of the tractor connection as a
prolonged dif\/ferential system in \cite{BEG} we see that a manifold is
almost Einstein if and only if there is a parallel section of the
standard tractor bundle. (A link between the Cartan connection and
Einstein metrics has been known for some time \cite{Sasaki}).
We establish in Section \ref{PES} that a
Poincar\'e--Einstein structure is simply a special case of such a
structure where the solution $x$ to \nn{prim} is a def\/ining function
for boundary, in fact a special def\/ining function in the sense of
Graham--Lee \cite{GL,GrSrni}. We shall show that along the boundary the
parallel tractor recovers the normal tractor from \cite{BEG} which
controls the relationship between the boundary and interior conformal
geometry; further details will follow in \cite{allein}. From this we
recover the (well known) total umbillicity of the boundary as an
immediate consequence. Finally in section \ref{model} we describe the
f\/lat model of this picture as a hyperbolic hemisphere with the equator
as boundary. In this setting we see, for example, that SO$(n+1,1)$,
which acts transitively on the interior of the homogeneous model,
arises as an isotropy subgroup of the conformal group acting on the
sphere; it is exactly the subgroup f\/ixing the parallel tractor.

In Section \ref{DNM} we discuss the specialisation to
Poincar\'e--Einstein manifolds of the Dirichlet--to--Neumann
machinery. Except in the simplest case (and there following Guillarmou
and Guillop\'e~\cite{GG}), the picture is incomplete and so this
should be viewed as an exploration and development of an emerging
picture. We see in Section \ref{scat} that the Laplacian operator
$(\Delta-s(n-s))$, which controls the scattering construction in
\cite{GrZ}, arises from a conformal operator on the almost Einstein
structure. On the other hand on the Poincar\'e--Einstein interior we
f\/ind that the GJMS conformal powers of the Laplacian of \cite{GJMS}
are compositions of such Laplacians (see expression \nn{scp}). This
suggests a relationship between the scattering construction and
Dirichlet--to--Neumann operators along the lines of \cite{BrGonon}, but
where the interior operator is a GJMS operator~$P_k$. The main link is
developed in Section \ref{adec}, we see in Proposition \ref{PEP} that
on a Poincar\'e--Einstein space, the solution space for these is a
direct sum of Laplacian (generalised) eigenspaces.

In Section \ref{trans} we sketch some new ideas for a
construction of Dirichlet--to--Neumann type operators between tensor
bundles. This is partly inspired by related ideas involving the
``curved translation principle'' for the construction of invariant
dif\/ferential operators, that arose in the pioneering article \cite{ER}
of Eastwood and Rice. This idea has been signif\/icantly developed and
adapted over the years \cite{Esrni,CSS,GSS,Sithesis}.

Finally it should be pointed out that many of the ideas and
contructions developed below generalise, with some weakening, to the
case of conformally compact manifolds which are only asymptotically
Einstein, or with further weakening to asymptotically hyperbolic
structures. Here we have specialised to Poincar\'e--Einstein structures
since in this case the picture seems especially appealing.

\section{Tractor calculus and hypersurfaces} \label{hypersec}

\newcommand{\cq}{\mbox{$\mathcal{Q}$}}

\newcommand{\J}{{\mbox{\sf J}}}

Let $M$ be a smooth manifold. To simplify the discussion we shall
assume throughout that $d\geq 4$ (with minor modif\/ications the
treatment can extended to include $d=3)$.  It will be convenient to
use some standard structures from conformal geometry, further details
and background may be found in \cite{CapGoamb,GoPetCMP}.  Recall that
a (Riemannian) {\em conformal structure\/} on $M$ is a smooth ray
subbundle $\cq\subset S^2T^*M$ whose f\/ibre over $p$ consists of
conformally related positive def\/inite metrics at the point
$p$. Sections of $\cq$ are metrics $g$ on $M$. So we may equivalently
view the conformal structure as the equivalence class $[g]$ of these
conformally related metrics.  The principal bundle $\pi:\cq\to M$ has
structure group $\mathbb{R}_+$, and so each representation $\mathbb{R}_+ \ni
t\mapsto t^{-w/2}\in {\rm End}(\mathbb{R})$ induces a natural line bundle
on $ (M,[g])$ that we term the conformal density bundle $\mathcal{E}[w]$.  In
general each vector bundle and its space of smooth sections will be
denoted in the same way.

We write $\bg$ for the {\em conformal metric}, that is the
tautological section of $S^2T^*M[2]:= S^2T^*M\otimes \mathcal{E}[2]$
determined by the conformal structure. This will be henceforth used to
identify $TM$ with $T^*M[2]$.  For example, with these conventions the
Laplacian $ \Delta$ is given by $\Delta=-\bg^{ab}\nabla_a\nabla_b=
-\nabla^b\nabla_b\,$ where $\nabla$ (or sometimes we will write $\nabla^g$)
is the Levi-Civita connection for some choice of metric $g$ from the
conformal class.  Note $\mathcal{E}[w]$ is trivialised by a choice of metric, and
we write~$\nabla$ (or again sometimes
$\nabla^g$) for the connection corresponding to this trivialisation.
It follows immediately that the (coupled) connection $ \nabla_a$ preserves the
conformal metric.  The conformal metric $\bg$ and its inverse will
henceforth be the default object used used to contract indices on
tensors even when we have f\/ixed a metric from the conformal class.

The Riemann curvature tensor
$R_{ab}{}^{c}{}_d$ is given by
\[(\nabla_a\nabla_b-\nabla_b\nabla_a)V^c=R_{ab}{}^{c}{}_d V^d,
\qquad\text{where} \quad \ V^c\in \mathcal{E}^c.\]
This can be decomposed into the totally trace-free {\em Weyl curvature}
$W_{abcd}$ and the symmetric {\em
Schouten tensor} $P_{ab}$ according to
\[
R_{abcd}=W_{abcd}+2g_{c[a}P_{b]d}+2g_{d[b}P_{a]c}.
\]
Thus $P_{ab}$ is
a trace modif\/ication of the Ricci tensor ${\rm
Ric}_{ab}=R_{ca}{}^c{}_b$. We write $J$ for the conformal metric trace
of $P_{ab}$.

For a given choice of metric $g$, the {\em tractor} bundle $\mathcal{T}$, or
using an obvious abstract index notation $\mathcal{T}^A$, may be identif\/ied
with a direct sum
\[
\mathcal{T}^A \stackrel{g}{=} \mathcal{E}[1]\oplus\mathcal{E}_a[1]\oplus\mathcal{E}[-1] .
\] Thus a section $U$ of $\mathcal{T}$ may be identif\/ied with a triple
$(\si,\mu_a,\rho)$; we will write simply $U^A= (\si,\mu_a,\rho)$.
The conformal transformation of these components is described in, for
example, \cite{BEG} and \cite{powerslap}.  From this, for example, one
sees that the map $\mathcal{T}^A \to \mathcal{E}[1]$ is conformally invariant and may
be regarded as a preferred element $X_A \in \Gamma\mathcal{T}_{A}[1]$ so that,
with $U^A$ again as above, we have $\sigma = U^{A}X_{A}$.  It also
describes the invariant injection $\mathcal{E}[-1] \to \mathcal{T}_{A}$ according to
$\rho \mapsto \rho X_A$. In computations, it is often useful to
introduce the remaining `projectors' from $\mathcal{T}^A$ to the components
$\mathcal{E}_a[1]$ and $\mathcal{E}[-1]$ which are determined by a choice of scale.
They are denoted by $Z_{Aa}\in\mathcal{E}_{Aa}[1]$ and $Y_A\in\mathcal{T}_A[-1]$,
where $\mathcal{E}_{Aa}[w]=\mathcal{T}_A\otimes\mathcal{E}_a\otimes\mathcal{E}[w]$, etcetera.

We describe any tensor product (or symmetric tensor product etcetera)
of the tractor bundle and its dual as tractor bundles.  If such a
bundle is tensored with some bundle of densities~$\mathcal{E}[w]$ then we
shall describe the result as a~{\em weighted tractor bundle}.
In
many cases we wish to indicate a~weighted tractor bundle without being
specif\/ic about the indices of the bundle or any symmetry these may
possess. Thus we write $\mathcal{T}^*[w]$ to mean a weighted tractor bundle
which is the tensor product of $\mathcal{E}[w]$ with any tractor bundle.
Finally, repeated tractor indices indicate a contraction, just as for
tensor indices.

The bundle $\mathcal{T}^A$ carries an invariant signature $(d+1,1)$ {\em
  tractor metric} $h_{AB}$, and a connection~$\nabla_a$ which
  preserves this.  For $ U^A$ as above, this metric is given by
\begin{gather}\label{trmet}
h_{AB}U^AU^B =2\rho \sigma +\mu^a\mu_a.
\end{gather}
 As a point on notation, we may also write
$h(U,U)$ for the expression in the display.  The tractor metric will
be used to raise and lower indices without further mention.  In terms
of the metric $g$ from the conformal class, the connection is given
explicitly by the following formula for $\nabla_a U^B$:
\begin{gather}\label{ndef}
    \nabla_a \left(\begin{array}{c} \sigma \\
  \mu^b \\ \rho \end{array}\right)
  = \left(\begin{array}{c} \nabla_a \sigma - \mu_a \\
  \nabla_a\mu^b + \delta_a{}^b \rho +
  P_{a}{}^b \sigma \\ \nabla_a\rho - P_{ac} \mu^c
  \end{array} \right) .
\end{gather}
Of course this may be extended to a connection on any tractor bundle
in the obvious way. The use of the same symbol $\nabla$ as for the
Levi-Civita connection is intentional. More generally, we shall use
$\nabla$ to mean the coupled Levi-Civita-tractor connection: this enables
us, for example, to apply $\nabla $ to weighted tractor bundles or
tensor-tractor. Although in this case it is not conformally invariant
it enables us to, for example, compute the covariant derivative of the
tractor projectors~$X$,~$Y$ and~$Z$.

As discussed in \cite{BEG},  there is  an invariant second order
operator between weighted tractor bundles due to T.Y. Thomas,
\[
D_A:\mathcal{T}^\ast[w]\to\mathcal{T}^\ast[w-1] ,
\]
by
\begin{gather}\label{Dform}
D_A V:=(d+2w-2)w Y_A V+ (d+2w-2)Z_{}^{a}\nabla_a V + X^A(\Delta-w\J)V,
\end{gather}
where $\J$ is the conformal metric trace of the Schouten tensor,
 i.e.\ $\J=\bg^{ab}P_{ab}$.
For an invariant construction of this see \cite{GoSrni99}.
Notice that, from the conformal invariance of $D$, it follows that the
tractor twisting of the Yamabe operator $\Box:=\Delta- (1-n/2)\J$ is
conformally invariant as an operator $\Box:\mathcal{T}^*[1-n/2]\to
\mathcal{T}^*[-1-n/2] $.  Thus as observed in \cite{GoSrni99}
one obtains conformal Laplacian operators as follows.
\begin{proposition}\label{lap}
The operators
\[
\Box_k:\mathcal{E}^\ast[\sfrac{k-d}{2}]\to \mathcal{E}^\ast[-\sfrac{k+d}{2}],\qquad \mbox{where $k\geq 2$ is even},
\]
defined by
\[
\Box_k:=D^{A} \cdots D^{B} \Box \underbrace{D_{B} \cdots
  D_{A}}_{(k-2)/2}
\]
are conformally invariant differential operators. These take
the form (up to a non-zero constant scale factor)
\[
\Delta^{k/2}+\mbox{ lower order terms},
\]
except when $d$ is even
and $d\leq k$.
\end{proposition}
The facts concerning the leading term follow easily by calculating
directly from the def\/inition, or there is a simple argument
essentially avoiding computation in \cite{GoSrni99}.  Note that, via
\nn{Dform}, this proposition gives an explicit formula for these
operators.

\newcommand{\om}{\omega}
\newcommand{\sbg}{\mbox{\boldmath{\scriptsize$ g$}}}
\subsection{Conformal hypersurfaces}\label{hyp}
Let us f\/irst recall some facts concerning a general hypersurface
$\Sigma$ in a conformal manifold $(M^d,[g])$.  By restriction of the
ambient conformal structure, a conformal structure is induced on
$\Sigma$. We write $(\Sigma^n,[g_\Sigma])$ ($n+1=d$) for this and
shall refer to it as the intrinsic conformal structure of $\Sigma$.
Note that the intrinsic conformal density bundles may be
identif\/ied in an obvious way with the restriction of the ambient
bundles carrying the same weight. We shall write
$\mathcal{E}^\Sigma[w]=\mathcal{E}[w]|_\Sigma$.

Since $d\geq 4$ we have $ n \geq 3$ and so the manifold $ \Sigma$ has
its own intrinsic tractor bundles, connections and so forth.
We shall denote the intrinsic tractor bundle of $\Sigma$
by $\mathcal{T}_\Sigma$. The relationship between this and $\mathcal{T}|_\Sigma$ can be
described in terms of a section of $\mathcal{T}|_{\Sigma}$ that we term the
normal tractor. Let $n_a\in \mathcal{E}_a[1]$ be a conormal f\/ield on $\Sigma$ such
that (along $\Sigma$) we have
$|n|^2_{\sbg}:=\bg^{ab}n_a n_b=1$. Note that this is conformally
invariant since $\bg^{-1}$ has conformal weight $-2$.  Now in the
scale $g$ (from $[g]$) the mean curvature of $\Sigma$ is given by
\[
H^g=\frac{1}{d-1}\big(\nabla_a n^a -n^an^b \nabla_an_b  \big),
\] as a conformal $-1$-density. This is independent of how $n_a$ is
  extended of\/f $\Sigma$.  Now under a~conformal rescaling, $g\mapsto
  \widehat{g}=e^{2\om} g$, $H$ transforms to $
  \widehat{H}=H+n^a\nabla_a\om$. This is exactly the transformation
  required so that
\[
N:\stackrel{g}{=}\left(\begin{array}{c}0\\
n_a\\
-H\end{array}\right),
\]
 is  a conformally invariant section $N$ of
  $\mathcal{T}|_\Sigma$.
Observe that, from \nn{trmet}, $h(N,N)=1$ along $\Sigma$. Obviously $N$ is
independent of any choices in the extension of $n_a$ of\/f
$\Sigma$. This is the {\em normal tractor} of \cite{BEG} and may be viewed as a tractor bundle analogue of the unit co-normal f\/ield from
the theory of Riemannian hypersurfaces.

Recall that a point $p$ in a hypersurface is an umbillic point if the
second fundamental form is trace free (with respect to the f\/irst
fundamental form) at $p$. This is a conformally invariant condition
and the hypersurface is totally umbillic if this holds at all points.
Dif\/ferentiating~$N$ tangentially along $\Sigma$ using the tractor connection,
 we obtain the following result directly from~\nn{ndef}.
\begin{proposition}\label{umbillic}
If the normal tractor $N$ is constant along a hypersurface $\Sigma$
then  the hypersurface $\Sigma$ is totally umbillic.
\end{proposition}
\noindent In fact constancy of $N$ along a hypersurface is equivalent to total
umbillicity. This is (Proposition~2.9) from~\cite{BEG}.

It is
straightforward to verify that the intrinsic tractor bundle to $
\Sigma$ may be identif\/ied with the conformally invariant subbundle
$\mathcal{T}^A_\Sigma$ of $ \mathcal{T}^A|_\Sigma$ which is orthogonal to the normal
tractor~$N^A$~\cite{BrGonon} (an observation which generalises, see \cite{Armstrong}). Thus we have
an invariant splitting
\[
\mathcal{T}^A|_\Sigma=\mathcal{T}^A_\Sigma \oplus {\cal N}^A
\]
given by
\[
v^A\mapsto (v^A-N^A N_Bv^B) +  N^A N_Bv^B.
\]
for $v^A\in\Gamma (\mathcal{T}^A)$. Of course this generalises easily to
tensor products of these bundles, and we shall always view the intrinsic
tractor bundles of $ \Sigma$ in this way; that is, as subbundles of the
restrictions to $\Sigma$ of ambient tractor bundles, the
sections of which are completely orthogonal to~$N^A$.  As a result,
we need
only one type of tractor index.

 We shall use the symbol $P_\Sigma$ to indicate the orthogonal
projection from any ambient weighted tractor bundle, restricted to
$\Sigma$, to the corresponding intrinsic-to-$\Sigma$ weighted tractor
bundle. For example, $P_\Sigma (\mathcal{T}^A|_\Sigma)=\mathcal{T}^A_\Sigma$. In fact
we shall henceforth drop the explicit restriction to $\Sigma$ and
regard this as implicit in the def\/inition of $P_\Sigma$. Thus we shall
write, for example, $P_\Sigma(\mathcal{T}_{AB}[w])=\mathcal{T}^\Sigma_{AB}[w]$; any
section $f_{AB}$ of this bundle has the property that
$f_{AB}N^A=0=f_{AB}N^B$.  The intrinsic-to-$\Sigma$ tractor-D operator
will be denoted $D^\Sigma_A$. We may similarly denote by $X_A^\Sigma$
the tautological tractor belonging to the intrinsic structure of
$\Sigma$. But note that $P_{\Sigma}(X_{A})=X_A|_\Sigma$, and it
follows from the def\/inition of $ X^\Sigma_A$ that in fact
$X^\Sigma_A=X_A|_\Sigma$.  A useful consequence of
these observations (and using the formula \nn{ndef}) is that if $f\in
\mathcal{T}^\ast_\Sigma[w]$, then, on $\Sigma$,
\begin{gather}\label{DXs}
D^A_\Sigma X_A f =(n +2w+2)(n+w)f=(d+2w+1)(d+w-1) f.
\end{gather}

\section[The conformal Dirichlet-Neumann operators of \cite{BrGonon}]{The conformal Dirichlet--Neumann operators of \cite{BrGonon}}
\label{DNBG}

\subsection{Boundary operators}\label{Bops}
 The Dirichlet--to--Neumann maps of \cite{BrGonon} are based around a
pair of boundary problems. The operators of Proposition \ref{lap} are
to be used for the interior operator, compatible with these we need
suitable boundary operators.  The basic prototype is the conformally
invariant Robin operator $ \delta : \mathcal{E}[w]\to \mathcal{E}[w-1] $ given in a
conformal scale $g$ by $\delta f= n^a\nabla^g_a f-wH^g f$ (e.g. \cite{cherrier}). In fact it
is easily verif\/ied that this is {\em strongly invariant}; twisting by
another connection does not destroy conformal invariance. In
particular we may twist this with tractor bundles by using the coupled
Levi-Civita-tractor connection in the formula for $\delta$:
\[
\delta: \mathcal{T}^*[w]\to \mathcal{T}^*[w-1]
\]
is conformally invariant. In fact on $\mathcal{T}^*[w]$, and for $w\neq 1-d/2$, we have
\begin{gather}\label{dD}
\delta = c\cdot N^A D_A
\end{gather}
for a non-zero constant $c$.

Further candidates for conformal boundary operators can be
proliferated using the machinery of the previous section, as
follows.

\medskip

\noindent{\bf Def\/inition.}
For each positive integer $\ell$
there is a conformally invariant dif\/ferential operator
  along $\Sigma$, $\delta_{\ell}$, which
  maps $\mathcal{T}^*[w]$ to
$\mathcal{T}^*_\Sigma [w-\ell]$, given by $\delta_1=\delta$, and
\[
\delta_{\ell} u= \left\{\begin{array}{l}
D_\Sigma^B\cdots D_\Sigma^A P_\Sigma(
\underbrace{D_A\cdots D_B}_{\ell/2}
u)
\mbox{ for }  2\leq \ell \mbox{ even},\\
D_\Sigma^B\cdots D_\Sigma^A P_\Sigma(\delta
\underbrace{D_A\cdots D_B}_{(\ell-1)/2}
u)
\mbox{ for }  3\leq \ell \mbox{ odd}.
\end{array}\right.
\]

For their use in boundary problems one needs information about the
order of the $\delta_\ell$ in directions transverse to $\Sigma$.
Suppose that $p\in \Sigma$ and in a neighbourhood of a point $p$,
$\Sigma$ is given by the vanishing of a def\/ining function $x$. We say
that a dif\/ferential operator $B: \mathcal{F}\to \mathcal{G}$ has {\em
normal order} $r_N$ {\em at} $p\in \Sigma$ if there exists a section
$\phi$ of $\mathcal{F}$ such that $B (x^{r_N} \phi)(p)\neq 0$ but for
any section $ \phi'$ of $\mathcal{F}$, $B (x^{r_N+1} \phi')(p)= 0$.
For our current purposes we only really need the $\delta_\ell$ as
follows.
\begin{proposition}
Let $k$ be a positive even integer. On a hypersurface $ \Sigma$ in a
manifold of dimension $d$
the conformally invariant differential operators
  along $\Sigma$,
\[
\delta_{\ell}:\mathcal{T}^*[\sfrac{k-d}{2}] \to \mathcal{T}^*_\Sigma [\sfrac{k-d-2\ell}{2}],
\]
have properties as follows. If $d$ is odd then
the $\delta_\ell$ have order and normal order $r=r_N=\ell$ for all
$\ell\in \mathbb{Z}_+$. If $d$ is even then the $\delta_\ell$ have
order and normal order $r=r_N=\ell$ if
$
 \ell +1 \leq k \leq d-2$ or $\ell+2\leq k= d.
$
\end{proposition}
This follows easily from the identity \nn{DXs} and the def\/inition of
of the tractor-D operator.

It turns out that appropriate combinations of the operators
$\delta_\ell$ lead to good elliptic problems (and in particular
problems which satisfy the so-called Lopatinski--Shapiro conditions
which signal well-posedness for boundary problems) with the operator
$\Box_k$ (from Proposition \ref{lap}), which itself is properly
elliptic. See Proposition~7.1 in \cite{BrGonon}.  A key but technical
point of that work with Branson is that there are modif\/ications of the
operators $\delta_\ell$ to similar conformal boundary operators $\delta'_{i}$
(each of the same respective normal order as $\delta_j$) so that we
may maintain the ellipticity properties but in addition  achieve
formally self-adjoint boundary problems. Let us simply summarise; the
reader is referred to \cite{BrGonon,grubb,kumano} for details and background.
\begin{proposition}\label{selfad} For $d$ even let
  $k\in\{0,2,\dots,d-2 \} $ and for $ d$ odd let $k\in 2\mathbb{Z}_+$.
  For each such $ k$
and each of ${\bf m}={\bf m}_{\rm D}=(0,2,\ldots,k-2)$,
${\bf m}={\bf m}_{\rm N}=(1,3,\ldots,k-1)$, and
${\bf m}={\bf m}_0:=(0,1,\ldots,k/2-1)$ there exist conformally
invariant normal boundary operators $\delta'_{\bf m}$
such that $(\Box_k,\delta'_{\bf m})$
is formally self-adjoint and satisfies the Lopatinski--Shapiro conditions.
\end{proposition}

\subsection[The Dirichlet-to-Neumann maps]{The Dirichlet--to--Neumann maps}

Let $M$ be an $d$-dimensional conformal
manifold of positive def\/inite metric signature,
with smooth boundary $\Sigma$.
Suppose that $k$ is even and, if $d$ is even, suppose that $k<d$.
Let ${\bf m}={\bf m}_{\rm D\,or\,N }\,$, and suppose that the
problem $(\Box_k\,,\delta'_{\bf m})$ has vanishing null space.
Take a density $u$ on $\Sigma$, and boundary data
\begin{gather}\label{SetBdryCond}
\delta'_{m_j}u=U_{\rm o},\qquad\delta'_{m_i}u=0\qquad \mbox{for all}\quad i\ne j,
\end{gather}
on $\Sigma$, where $j$ is a chosen element of $\{1,\ldots,k/2\}$.  Let
$E_{k,m_j}$ be the solution operator for the system $\Box_ku=0$ with
\nn{SetBdryCond}; by elliptic regularity the range of $E_{k,m_j}$ is
smooth and by construction it is an invariant operator carrying
$\mathcal{E}_\Sigma[\sfrac{k-d-2m_j}2]$ to $\mathcal{E}[\sfrac{k-d}2]$.  We can now
take $E_{k,m_j}u$ and apply $\delta'_\ell$ (or $\delta_\ell$).
($\ell$ need not be one of the normal orders in ${\bf m}$.)
Composing,
\begin{gather}\label{comp}
\mathcal{E}_\Sigma[\sfrac{k-d-2m_j}2]\stackrel{E_{k,m_j}}
{\tlongrightarrow}
\mathcal{E}[\sfrac{k-d}2]
\stackrel{\delta'_{\ell}}{\tlongrightarrow}
\mathcal{E}_\Sigma[\sfrac{k-d-2\ell}2],
\end{gather}
we obtain invariant operators
\[
P_{k,{\bf m},m_j\,,\ell}:\mathcal{E}_\Sigma[\sfrac{k-d-2m_j}2]\to
\mathcal{E}_\Sigma[\sfrac{k-d-2\ell}2].
\]
Note that for this construction to make sense, as given, we need the
{\em source problem} $E_{k,m_j}: \mathcal{E}_\Sigma[\sfrac{k-d-2m_j}2]\to
\mathcal{E}[\sfrac{k-d}2]$ to be uniquely solvable. The second part of the
construction $\mathcal{E}[\sfrac{k-d}2] \to \mathcal{E}_\Sigma[\sfrac{k-d-2\ell}2]$ is
related to a complementary {\em target problem}, but its solvability
properties are not required.

From elementary representation theory and the invariance of the
construction, it is straightforward to show that when $m_j+\ell\ne
k-1$ the operators $P_{k,{\bf m},m_j\,,\ell}$ vanish for the standard
conformal class on the unit ball. Equivalently they vanish on the unit
hemisphere, which is in the conformal class of the unit ball, and this
is a convenient homogeneous setting for exploiting the spherical
harmonics in order to study this and related issues \cite[Theorem
8.4]{BrGonon}. In particular there one also sees that $\pkm$ has
principal part $(-\Delta_\Sigma)^{(k-1-2m_j)/2}$, up to multiplication
by a non-zero universal constant. In addition, by construction, the
operators $\pkm $ are formally self-adjoint on any conformal
manifold \cite[Theorem 8.5]{BrGonon}.

Writing $P_{k,{\bf m},m_j}:= \pkm$, in summary we have the following:
\begin{theorem}\label{psidops} Let $M^{d=n+1}$ be a conformal
manifold of positive definite metric signature,
with smooth boundary $\Sigma$.
Suppose that $k$ is even and, in case $d$ is even, suppose that $k<d$.
Let ${\bf m}={\bf m}_{\rm D\,or\,N,or\,0}\,$, and suppose that the
problem $(\Box_k\,,\delta'_{\bf m})$ has vanishing null space.
Then there exist canonical
conformally invariant
operators
\[
P_{k,{\bf m},m_j}: \ \mathcal{E}_\Sigma[\sfrac{k-n-2m_j-1}2]\to
\mathcal{E}_\Sigma[\sfrac{-k-n-2m_j+1}2], \quad \quad m_j\in {\bf m},
\]
with principal part $\Delta^{(k-1-2m_j)/2}$.
\end{theorem}

\section[Poincar\'e-Einstein Manifolds and generalisations]{Poincar\'e--Einstein Manifolds and generalisations}\label{PEgen}

\newcommand{\II}{I \hspace*{-3pt} I}

We give here a conformal development of Poincar\'e--Einstein
manifolds.

\subsection{Almost Einstein manifolds}\label{alE}

The Schouten tensor $P$ (or $P^{g}$), introduced earlier, is related to the
 Ricci tensor by
\[
\Ric=(d-2)P+J g,
\] where $J$ is the metric trace of $P$. The metric $g$ is conformally
Einstein if and only if there is a~non-vanishing solution $x\in
C^\infty(M)$ to the equation
\begin{gather}\label{prim}
\text{trace}-\text{free}(\nabla\nabla x + P x)=0;
\end{gather}
if $x$ is such a solution then it follows easily from the conformal
transformation of $P$ that the metric $\widehat{g}=x^{-2}g$ is
Einstein \cite{BEG}.

Note that the requirement that $x$ be non-vanishing is critical if we
want a solution $x$ to be a genuine conformal factor:
$\widehat{g}=x^{-2}g$ will blow up conformally at points where $x$
vanishes. Nevertheless, let us relax this and allow any
solution. Following \cite{Goalmost} we will say that $(M,g)$ is {\em
almost Einstein} if there is a solution $x\in C^\infty(M)$ to the
equation \nn{prim}.

It turns out that the almost Einstein condition is a useful
weakening of the Einstein equations. First observe that the equation
\nn{prim} is conformally well behaved. If we replace $x$ with a
conformal density of weight 1, $\si\in \mathcal{E}[1]$, then it is easily verif\/ied that
\nn{prim} is actually conformally invariant; it descends to a well
def\/ined equation on the conformal structure $(M,[g])$. In a scale we
may write the equivalent equation on $\si$ in the form
\begin{gather}\label{primc}
\nabla_a\nabla_b \si + P_{ab}\si +\bg_{ab}\rho =0,
\end{gather}
 where the
section $\rho\in \mathcal{E}[-1]$ captures the trace part. The next crucial
observation is that by inspection of the formula \nn{ndef} we see that
the equation for a parallel section of standard tractor bundle $\mathcal{T}$
is just the prolonged system for this equation. Informally stated the
f\/irst equation from \nn{ndef} equates the variable $\mu_a$ to the
derivative of $\si$. Then the middle equation from \nn{ndef} just is
$\nabla_a\nabla_b \si + P_{ab}\si +\bg_{ab}\rho =0$, while
$\nabla_a \rho=P_{ac}\mu^c$ is a dif\/ferential consequence of this.
From that system it follows that if
$I\stackrel{g}{=}(\si,\mu_a,\rho)$ is a parallel section for
$\nabla^{\mathcal{T}} $ then necessarily
\begin{gather}\label{pD}
 \big(\si, \mu_a, \rho\big) =
\left(\si,\nabla_a \si , \frac{1}{d}(\Delta \si -\J \si)\right).
\end{gather}
That is
$I=\frac{1}{n}D_A\si$, where $D$ is Thomas tractor-D operator introduced
in \nn{Dform}.
Now by construction $D$ is dif\/ferential and, on the other
hand, parallel transport between two points along any curve gives an
isomorphism of the vector bundle f\/ibres over those points. It follows
immediately that if $I$ is parallel and $\si$ vanishes on any
neighbourhood then $I$ vanishes everywhere. Equivalently, if $I \neq
0$ is parallel then $\si:=h(X,I)$ is non-vanishing on an open dense
set; this is the key.
Summarising,  we have the following.
\begin{theorem}\label{key}
An almost Einstein structure is a conformal manifold $(M,[g])$
 equipped with a~paral\-lel (standard) tractor $I\neq 0$. The mapping
 from non-trivial solutions of \eqref{primc} to parallel tractors is by
 $\si\mapsto \frac{1}{n}D\si$ with inverse $I\mapsto \si:=h(I,X)$, and $\si$ is
 non-vanishing on an open dense set $M\setminus \Sigma$. On this set
 $\g:=\si^{-2}\bg$ is Einstein.
\end{theorem}

From the theorem we see that an almost Einstein manifold just {\em is}
a conformal manifold with a parallel standard tractor $I$ and we write
$(M,[g],I)$ to indicate this.  The set $\Sigma$, where the almost
Einstein ``scale'' $\si=h(X,I)$ vanishes, is called the scale
singularity set. Although it is not essential for our current
discussion, we note that on Riemannian signature manifolds the
possibilities for this are severely restricted as follows
\cite{allein}. We write $|I|^2$ as a shorthand for~$h(I,I)$.
\begin{theorem}\label{classthm} Let $(M,[g],\si)$ be an almost Einstein
structure and write $I:=\frac{1}{n}D\si$.  If $|I|^2<0$ then $\Sigma$
 is empty and $(M,\si^{-2}\bg)$ is Einstein with positive scalar
 curvature; If $|I|^2=0$ then $\Sigma$ is either empty or consists of
 isolated points, and $(M\setminus \Sigma,\si^{-2}\bg)$ is Ricci-flat;
 if $|I|^2> 0$ then the scale singularity set $\Sigma$ is either
 empty or else is a totally umbillic hypersurface, and $(M\setminus
 \Sigma, \si^{-2} \bg)$ is Einstein of negative scalar curvature.
\end{theorem}
For the special case of Poincar\'e--Einstein manifolds, we shall see
the result concerning total umbillicity in Corollary \ref{umC} below.
Almost Einstein metrics turn up in the classif\/ications by Derdzinski
and Maschler of K\"ahler metrics which are almost everywhere conformal to
Einstein, see e.g. \cite{DM} and references therein.

\subsection[Poincar\'e-Einstein spaces]{Poincar\'e--Einstein spaces} \label{PES}

Recall that a Poincar\'e--Einstein structure is a compact manifold
$(M^{d=n+1},g)$ with boundary $\Sigma=\partial M$. There is a def\/ining
function $x$ for $\Sigma$ so that $(M_+,\g)$ is Einstein with
$\Ric^{\g}=-n \g$, where $M_+:= M\setminus \Sigma$ and, on this,
$g_+:=x^{-2}g$.

\begin{proposition}\label{peae}
Poincar\'e--Einstein manifolds are scalar negative almost Einstein
structures.  Conversely on a compact manifold $M$ with boundary
$\Sigma$, an almost Einstein structure $(M,[g],I)$ with $|I|^2=1$, and
such that the scale singularity set is  the boundary $\Sigma$, is a
Poincar\'e--Einstein metric.
\end{proposition}

\begin{proof}
If $M$ is a Poincar\'e--Einstein manifold then, by def\/inition, there is
a def\/ining function $x$ for the boundary so that $\g:=x^{-2}g$ is
Einstein.  Thus $x$ is a smooth function that solves \nn{prim} on the
interior and so, by continuity, also to the boundary. This gives the
result.  In the tractor picture we may equivalently observe that, on
the interior, $I:=\frac{1}{n}D\si$ is parallel with $\si = x\tau$
where $\tau\in \mathcal{E}[1]$ is the scale giving $g$, that is
$g=\tau^{-2}\bg$. So we have $\g=\si^{-2}\bg$. By continuity $I$ is
parallel to the boundary. Obviously the function $h(I,I)$ is constant.
Of\/f the zero set $\Sigma$, and calculating in the scale
$\g=\si^{-2}\bg$, we have $\nabla^{\g}\si=0$. Thus from \nn{trmet} and
\nn{pD} we have $|I|^2=-\frac{2}{d}J^{\g}$ where $J^{\g}$ is the
$\g$-trace of the Schouten tensor $P^{\g}$. From the relationship
between the Schouten and Ricci tensors it follows immediately that the
normalisation $\Ric (\g)=-n \g$ is exactly the condition $|I|^2=1$.

The converse direction is essentially clear from the last
observations. If $I$ is parallel with $|I|^2=1$ then of\/f the zero set
of $\si:=h(X,I)$ we have $\Ric (\g)=-n \g$ where $\g:=\si^{-2}\bg$. On
the other hand along $\Sigma$ we have $\si=0$ and so, since $|I|^2=1$,
it follows from \nn{pD} and \nn{trmet} that $\bg^{-1}(\nabla \si,\nabla
\si)=1$. In particular $\nabla \si$ is non-vanishing along $\Sigma$ and
so $\si$ is a ``def\/ining density'' for~$\Sigma$. Choosing a metric $g$
for $M$ we have $g=\tau^{-2}\bg$ for some non-vanishing weight 1
density~$\tau$. We set $x:=\si/\tau$ and note that $\Sigma$ is the
zero set of~$x$. Since $\Sigma$ is the boundary of $M$ (by a~suitable
sign choice for $\tau$) we may assume without loss of generality that
$x$ is a non-negative function. In terms of $g$, the result
$\bg^{-1}(\nabla \si,\nabla \si)=1$ is equivalent to $|dx|_{g}^2=1$, and so
$x$ is a def\/ining function for the boundary $\Sigma$. (In fact $x$ is
a special def\/ining function in the sense of \cite[Lemma 2.1]{GrSrni}
and~\cite{GL}.)  On the other hand the Einstein metric $\g$ is
$x^{-2}g$ and this completes the case.
\end{proof}

From the Proposition we see that by specifying the almost Einstein
structure $(M,[g],I)$ we have an essentially conformal description of
a Poincar\'e--Einstein structure. We next see that~$I$ encodes more
than simply the Einstein scale. Recall the notion of a~normal tractor,
for a~hypersurface or boundary, as introduced in Section \ref{hyp}.
\begin{proposition}\label{IvsN} Let $(M,[g],I)$ be a Poincar\'e--Einstein
manifold.
Along the boundary $\Sigma$ we have $I_A=N_A$ where $N_A$ is the
normal tractor for $\Sigma$.
\end{proposition}
\begin{proof}
First note that since $I_A$ has (conformally
invariant) length 1 everywhere this is in particular true along
$\Sigma$. (Of course $N_A$ has this property along $\Sigma$.)

Now
\[
I_A=\frac{1}{d}D_A\si\stackrel{g}{=}
\begin{pmatrix} \si \cr  \nabla_a \si\cr \frac{1}{d}(\Delta \si - \J \si)
\end{pmatrix} .
\]
 Let us write $n_a := \nabla_a \si$.  Along
$\Sigma$ we have $\si=0$, and so
\[
I_A|_\Sigma  \stackrel{g}{=} \begin{pmatrix} 0 \cr  n_a \cr
 \frac{1}{d}\Delta \si  \end{pmatrix}.
\] Note that from \nn{trmet} and $|I|^2=1$ we have that
$\bg^{ab}n_an_b=1$ on~$\Sigma$, and $n_a$ is seen to be a weight~1
unit co-normal for $\Sigma$.

Next we calculate the mean curvature $H$ in terms of $\si$.  Recall
that the second fundamental form of $\Sigma$ is $\II_{ab}=\Pi^c_a\Pi^d_b
\nabla_c n_b$ (along $\Sigma$) where $\Pi$ is the orthogonal projection
operator given by
\[
\Pi^c_a=\delta^c_a-n^cn_a .
\]
By construction this is independent of how $n_a$ is extended of\/f $\Sigma$.
Thus along $\Sigma$ we have
\[
(d-1)H=\nabla^an_a-n^an^b\nabla_b n_a .
\]
Now since $I_A$ is parallel we have
\[
\nabla_b n_a =\nabla_b\nabla_a \si = -P_{ab}\si - \frac{1}{d}\bg_{ab}(\Delta \si-\J \si).
\]
Along $\Sigma$ this simplif\/ies to $\nabla_b n_a=  -\frac{1}{d}\bg_{ab}\Delta \si$, and
so $\nabla^an_a=-\Delta \si$. Thus we have
\[
(d-1)H=-\left(1-\frac{1}{d}\right)\Delta \si\quad \Rightarrow \quad H
=-\frac{1}{d}\Delta \si ,
\]
whence
\[
I_A|_\Sigma  \stackrel{\tilde{g}}{=} \begin{pmatrix} 0 \cr  n_a \cr
 -H  \end{pmatrix} ,
\]
as required.
\end{proof}

Using Proposition \ref{umbillic}, the following is an immediate
consequence of the Proposition.
\begin{corollary} \label{umC}
If $(M,[g],I)$ is a Poincar\'e--Einstein manifold then  the boundary $\Sigma=\partial M$ is
totally umbillic.
\end{corollary}

\subsection[The Poincar\'e-Einstein model space]{The Poincar\'e--Einstein model space} \label{model}

We shall construct here a model for the Poincar\'e--Einstein space, a
model which f\/its with the conformal picture developed above.

Consider $\mathbb{R}^{d+2}$ equipped with a non-degenerate bilinear
form $\mathcal{H}$ of signature $(d+1,1)$. The {\em null cone}
$\mathcal{N}$ of zero-length vectors forms a quadratic variety and
(projectivising this picture) the corresponding quadric in
$\mathbb{P}_{d+1}$ is topologically a sphere $\mathbb{S}^d$.  Let us
write $\mathcal{N}_+$ for the forward part of $\mathcal{N}\setminus
\{0 \}$ and write $\pi$ for the natural submersion $\mathcal{N}_+\to
\mathbb{S}^d$.  Each point $p\in \mathcal{N}_+$ determines a positive
def\/inite inner product on $T_{x=\pi{p}}\mathbb{S}^d$ by
$g_x(u,v)=\mathcal{H}_p(u',v')$ where $u',v'\in T_p\mathcal{N}_+$ are
lifts of $u,v\in T_x\mathbb{S}^d$. For a given vector $u\in T_x
\mathbb{S}^{d}$ two lifts to $p\in \mathcal{N}_+$ dif\/fer by a vertical
vector f\/ield. Since any vertical vector is normal (with respect to
$\mathcal{H}$) to the cone it follows that $g_x$ is independent of the
choices of lifts. Clearly then, each section of $\pi$ determines a
metric on $\mathbb{S}^d$ and by construction this is smooth if the section
is. Now, viewed as a metric on $T\mathbb{R}^{d+2}$, $\mathcal{H}$ is
homogeneous of degree 2 with respect to the standard Euler vector
f\/ield $E$ on $\mathbb{R}^{d+2}$, that is $\mathcal{L}_E \mathcal{H}=2
\mathcal{H}$, where $\mathcal{L}$ denotes the Lie derivative. In
particular this holds on the cone, which we note is generated by $E$.
On the other hand if a vector f\/ield on $\mathbb{R}^{d+2}\setminus \{
0\}$ is the lift of a vector f\/ield on $\mathbb{S}^d$ then it is
necessarily homogeneous of degree 0.  It follows that that two
dif\/ferent sections of $\pi:\mathcal{N}_+\to \mathbb{S}^d$ determine
conformally related metrics.  (We have stated this globally, but of
course the same holds locally over neighbourhoods of $\mathbb{S}^d$.)
We will see shortly that the standard sphere metric is in the
conformal class. Thus $\mathbb{S}^d$ is equipped canonically with the
standard conformal structure for the sphere (but not with a preferred
metric from this class).  This is the standard model of a homogeneous
``f\/lat'' conformal structure. Evidently we may identify
$\mathcal{N}_+$ as the (total space) of the $\mathbb{R}_+$-ray-bundle
of metrics over $\mathbb{S}^d$; the bundle consisting of metrics from
this conformal class.

We can construct a Poincar\'e--Einstein metric over a cap of the
sphere $\mathbb{S}^d$ as follows.  Take a~covector $I\in (\mathbb{R}^{d+2})^*$ of
length 1 and by the standard parallel transport (of $\mathbb{R}^{d+2}$ viewed
as an af\/f\/ine structure) view this as a constant section of
$T^*\mathbb{R}^{d+2}$. Then, writing $X^A$ for the standard coordinates on
$\mathbb{R}^d$, the intersection of the hyperplane $I_AX^A=1$ with $\mathcal{N}_+$,
which we shall denote~$S_+$, is a section of $\pi$ over an open cap
$C_+$ of the sphere.  Let us write $\g$ for the metric $S_+$ gives on
$C_+$. On the other hand the hyperplane $I_AX^A=0$ (parallel to the
previous) intersects $\mathcal{N}_+$ in a cone of one lower dimension. The
image of this under $\pi$ is a copy of $\mathbb{S}^{n}$ embedded in $\mathbb{S}^d$
(where as usual $d=n+1$). With respect to the given manifold structure
on~$\mathbb{S}^d$, this $\mathbb{S}^n$ is a boundary for its union with $C_+$ which we
denote by $C$. This follows because any null line though the origin
and parallel to the $I_AX^A=1$ hyperplane lies in the hyperplane
$I_AX^A=0$, whereas every other null line through the original meets
the $I_AX^A=1$ hyperplane. The open cap $C_+$ parametrises those null
lines which meeting this hyperplane in the forward null cone (i.e. at
a point of $S_+$).  Note that the boundary $\mathbb{S}^{n}=\partial C$
canonically has no more than a conformal structure. This may obviously
be viewed as arising as a restriction of the conformal structure
on $\mathbb{S}^{d}$. Equivalently we may view its conformal structure as
arising in the same way as the conformal structure on $\mathbb{S}^d$, except
in this case by the restriction of $\pi$ to the sub-cone $I_AX^A=0$ in
$\mathcal{N}_+$, and from $\mathcal{H}$ along this sub-cone.  Note that any
metric from the conformal class on $\mathbb{S}^d$ determines a metric on~$C$
by restriction. Denoting one such by $g$ note that on $C^+$ this is
conformally related to~$\g$.

Write $\bg$ for the restriction of $\mathcal{H}$ to vector f\/ields in
$T\mathcal{N}_+$ which are the lifts of vector f\/ields on~$\mathbb{S}^d$. Then for any
pair $u,v\in \Gamma(T\mathbb{S}^d)$, with lifts $u'$, $v'$, $\bg(u',v')$ is a
function on $\mathcal{N}_+$ homogeneous of degree 2, and which is independent
of how the vector f\/ields were lifted. Since $\mathcal{N}_+$ may be identif\/ied
with the (total space of the) bundle of conformal metrics,
$\bg(u',v')$ may be identif\/ied with a conformal density of weight $2$
on $\mathbb{S}^d$. This construction determines a section of
$S^2T^*\mathbb{S}^d\otimes E[2]$ that we shall also denote by $\bg$. This is
the usual conformal metric for the sphere. Let us henceforth identify,
without further mention, each function on $\mathcal{N}_+$ which is homogeneous
of degree $w\in \mathbb{R}$ with the corresponding conformal density of weight
$w$. With $\si:=I_AX^A$, as above, note that $\si^{-2}\bg$ is
homogeneous of degree 0 on $\mathcal{N}_+$ and agrees with the restriction of
$\mathcal{H}$ along~$S_+$. Thus on $C_+$, $\si^{-2}\bg=\g$, the metric
determined by~$S_+$.  Similarly on $C$ we have $g=\tau^{-2}\bg$, where
$\tau$ is a non-vanishing conformal density of weight 1.  So on $C_+$,
$\g=x^{-2}g$ where $x$ is the {\em function} $\si/\tau$.

We may now put these observations into a more general context via the
tractor bundle on~$\mathbb{S}^d$. Let us write $\rho^t$ for the natural action
of $\mathbb{R}_+$ on $\mathcal{N}_+$ and then $\rho^t_*$ for the derivative of
this. Now modify the latter action on $T\mathbb{R}^{d+2}$ by rescaling: we
write $t^{-1}\rho^t_*$ for the action
of $\mathbb{R}_+$ on $T\mathbb{R}^{d+2}$ which takes $u\in T_p\mathbb{R}^{d+2}$
to $t^{-1}(\rho_*^t u)\in T_{\rho^t(p)}\mathbb{R}^{d+2}$. Note that $u$ and
$t^{-1}(\rho_*^t u) $ are parallel, according to the standard af\/f\/ine
structure on $\mathbb{R}^{d+2}$. It is easily verif\/ied that the quotient of
$T\mathbb{R}^{d+2}|\mathcal{N}_+$ by the $ \mathbb{R}_+$ action just def\/ined is a rank $d+2$
vector bundle $\mathcal{T}$ on $M$. Obviously the parallel transport of
$\mathbb{R}^{d+2}$ determines a parallel transport on $\mathcal{T}$, that is a
connection $\nabla$. Since $\mathbb{R}^{d+2}$ is totally parallel this
connection is f\/lat. The twisting of $\rho^t_*$ to $t^{-1}\rho^t_*$ is
designed so that the metric $\mathcal{H}$ on $\mathbb{R}^{d+2}$ also descends
to give a (signature $(d+1,1)$) metric $h$ on $\mathcal{T}$ and clearly this
is preserved by the connection.  In fact $(\mathcal{T},h,\nabla)$ is the usual
normal standard tractor bundle. This is proved under far more general
circumstances in \cite{CapGoamb} (see also \cite{GoPetCMP}); it is
shown there that the tractor bundle arises from the Fef\/ferman--Graham
ambient metric by an argument generalising that above. In this picture
the Euler vector f\/ield $E=X^A\partial/\partial X^A$ (using the
summation convention), which generates the f\/ibres of $\pi$, descends
to the canonical tractor f\/ield $X\in \mathcal{T}[1]$.

It follows from these observations that, since the vector f\/ield $I$ is
parallel on $\mathbb{R}^{d+2}$, its restriction to $\mathcal{N}_+$ is equivalent to a
parallel section of $\mathcal{T}$; we shall also denote this by $I$. So this
is an almost Einstein structure on $\mathbb{S}^d$ and hence (by restriction)
on $C$; $|I|^2=1$ means that the almost Einstein structure we recover
has $\Ric(\g)=-n \g$ on $C_+$.  Evidently the conformally invariant
``top slot'' of the tractor $I$ is $\si = X^AI_A=h(I,E)$ (as a
homogeneous function -- but as mentioned above homogeneous functions
on $\mathcal{N}$ may be identif\/ied with conformal densities on $C$). The zero
set $\Sigma$ for this is exactly $\mathbb{S}^n=\partial C$; recall any null
line of $\mathcal{N}_+$ that does not lie in the $I_AX^A=0$ hyperplane meets
the $I_AX^A=1$ hyperplane (where, viewing $\si$ as a homogeneous
function, we have $\si=1$). So now it follows from Proposition
\ref{peae} that $(C,[g],I)$ is a~Poincar\'e--Einstein space.

Of course this result may easily be verif\/ied by direct calculation,
but such a calculation would only obscure this simple geometric
picture.  Note, for example, that Proposition \ref{IvsN} is realised
geometrically here as the fact that the same constant ambient vector
$I_A$ def\/ines both the section~$S_+$, that gives the metric on $C_+$,
and the boundary manifold $\Sigma$ via the subcone where $\mathcal{N}_+$ meets
$I_AX^A=0$.

Finally we should say that in fact $I$ and the cone $\mathcal{N}$ determine a
scalar negative almost Einstein structure on the entire sphere
\cite{allein}.  The Poincar\'e-metric is just the $\si\geq 0$ part of
this.  A similar construction using a constant vector with norm
satisfying $|I|^2=-1$ gives the standard (Einstein) metric on the
sphere, while a Euclidean (i.e.\ metric f\/lat away from the scale
singularity point) almost Einstein structure on the sphere minus a
point is obtained by using a null constant vector~$I$. Note that the
group ${\rm SO}(h)~\cong ~{\rm SO}(d+1,1)$ acts transitively on
$\mathcal{N}_+$. Thus, in each case, the isotropy group of the constant vector
$I$ is a subgroup of~${\rm SO}(h)$ which acts transitively on the
corresponding almost Einstein interior. For example in the original
case, with $|I|^2=1$, the isotropy group is isomorphic to SO$(d,1)$
and acts transitively on the hyperbolic cap $C_+$.
 Similarly in the other cases: if we take $I$ such that $|I|^2=-1$
then the isotropy group is a copy of ${\rm SO}(d+1)$ which acts
transitively on the copy of the sphere that $I$ determines; for null
$I$ the isotropy group is the Euclidean group f\/ixing $I$ (and is a
subgroup of the parabolic stabilising the line generated by $I$).

\section[The Dirichlet--Neumann machinery on Poincar\'e--Einstein manifolds]{The Dirichlet--Neumann machinery\\ on Poincar\'e--Einstein manifolds} \label{DNM}

Here we wish to discuss the specialisations of certain key
dif\/ferential operators to Poincar\'e--Einstein manifolds and the implications
for a Dirichlet--to--Neumann construction
along the lines of that in \cite{BrGonon}. Throughout this section we take
$(M^d,[g],I)$ to be a Poincar\'e--Einstein manifold  with boundary
$\Sigma^n =\partial M$. We write $\bg$ for the conformal metric on $M$
and $\si\in \mathcal{E}[1]$ for the almost Einstein scale $h(X,I)$, that is
$\g:=\si^{-2}\bg$ is Einstein of negative scalar curvature on the
interior $M_+$.

First we make some elementary observations concerning the conformal calculus.

\subsection{The scattering Laplacian}\label{scat}

Recall that $I_A= \frac{1}{d}D_A\si$ is the parallel tractor
corresponding to the Einstein scale. First observe that from \nn{dD}
and Proposition \ref{IvsN} it follows that $I^AD_A$ gives the
conformal Robin operator $\delta$ along~$\Sigma$ (at least on
densities or tractors of weight $w\neq 1-d/2$). Here we expose the
further role of~$I^AD_A$.

For $u\in\mathcal{E}[w]$ in $M$ we wish to calculate $I^A D_A u$ on the interior $M_+$.
In particular let us express this in terms of  the interior Einstein metric
$\g$. We have $\nabla^{\g} \si=0$, and so
\[
I_A D^A u = \si \left(\begin{array}{ccc} -\J/d & 0 & 1 \end{array} \right)
 \left(\begin{array}{c} w (d+2w-2) u \\
(d+2w-2)\nabla^{\g} u \\
\Delta u - w \J u \end{array} \right),
\]
where on the right-hand-side $I$ and $D^A$ are expressed in terms of
the metric $\g$, but we are still allowing the tensorial objects to be
density valued. So we obtain
\[
I^AD_A u =\si \left(\Delta u-\frac{2}{d}\J(d+w-1)w u \right).
\]

Now let us use $d=n+1$ to replace $d$ and,
for reasons that will shortly be clear, set
 $s:=n+w$.
Then  we obtain
\begin{gather}\label{preslap}
I^AD_A u = \si \left(\Delta+\frac{2\J}{d}s(n-s)\right)u .
\end{gather}
Here $\Delta $ and $\J$ are (density-valued with weight $(-2)$) and
$\Delta^{\g}= \si^2 \Delta$, $J^{\g}=\si^{2}\J$. But for the Einstein metric
$\g$ with $\Ric(\g)=-n\g$ we have
 $J^{\g}=-d/2$ and so  this simplif\/ies to
\begin{gather}\label{slap}
\si I^AD_A u =  (\Delta^{\g}-s(n-s))u ,
\end{gather}
which agrees with the Laplacian controlling the scattering construction of
Graham--Zworski, see~(3.2) of \cite{GrZ}. For convenience let us refer
to this as the {\em scattering Laplacian}.

\subsection[The GJMS operators on a Poincar\'e-Einstein manifold]{The GJMS operators on a Poincar\'e--Einstein manifold}

Using the
Fef\/ferman--Graham ambient metric Graham--Jenne--Mason--Sparling (GJMS)
constructed in \cite{GJMS} a large family of  conformally invariant
operators $P_k$ between density bundles. These take the form
\[ P_k= \Delta^{k/2}+ \text{lower~order~terms}.
\] In fact $P_2$ is the usual conformal Laplacian from physics and, in
the case of Riemannian signature, as we have here, is often termed the
Yamabe operator. $P_4$ is due to S.~Paneitz, while $P_6$ was
constructed by V.~W\"unsch. Except at these low orders, the explicit
details of the GJMS operators are complicated \cite{GoPetCMP}, and no
general formula is available. For these and other related reasons the
operators $\Box_k$, with their explicit and manifestly formally
self-adjoint formulae as in Proposition~\ref{lap}, were preferred for
the construction of Dirichlet--to--Neumann operators in \cite{BrGonon}.

However the GJMS operators simplify dramatically on Einstein
manifolds; from \cite[Proposition 7.9]{FGnew} or \cite[Theorem
1.2]{powerslap} on an Einstein manifold $(M^d,\g)$ we have
\begin{gather}\label{powers}
P_{k}= \prod_{\ell=1}^{k/2}(\Delta^{\g} + \lambda_\ell),
\end{gather}
where $\lambda_\ell= {\rm Sc}^{\g} (d+2\ell-2)(d-2\ell)/(4d(d-1))$ and
$\Delta^{\g}=-\nabla^a\nabla_a$ is the Laplacian for $\g$.  This generalises
the situation on the sphere, as observed some time ago by Branson
\cite{Tomsharp}.

It will be useful for us to know where \nn{powers} comes from in the
tractor picture.  On a conformally Einstein manifold $(M,[g])$ if
$\si$ is an Einstein scale, with corresponding metric $\g$ and
parallel tractor $I_A:=\frac{1}{d}D_A \si$, then we may form the
operator $P^{\g}_k:\mathcal{E}[\frac{k-n}{2}]\to \mathcal{E}[-\frac{k+n}{2}]$ ($k\in
2\mathbb{N}$) by
\begin{gather}\label{defP}
P^{\g}_k u= \si^{1-k/2} I^{A_2}\cdots I^{A_{k/2}}
\Box D_{A_2}\cdots D_{A_{k/2}} u .
\end{gather}
By construction this depends on $\g$. Surprisingly if $\g'$ is another
Einstein metric in the conformal class $\g$ then $P^{\g}_k=P^{\g'}_k$
(see \cite[Theorem 3.1]{powerslap}). So for any $k\in 2\mathbb{N}$,
$P_k^{\g}$ is a canonical operator on conformally Einstein
manifolds. Thus if the conformal class is f\/ixed we may omit the $\g$
in $P^{\g}_k$. This is not a clash of notation since in fact on
conformally Einstein manifolds, and for the $k$ where the GJMS
operators $P_k$ are def\/ined, they agree with the operator \nn{defP} \cite[Theorem 3.3]{powerslap}.

The factorisation in \nn{powers} arises as follows. Using
that $I$ is parallel, we may rewrite \nn{defP} as
\[
P_k u=
\si^{-k/2} (I^{A_{k/2}}D_{A_{k/2}})\circ \cdots \circ  (I^{A_2}D_{A_2})
\circ (I^{A_1}D_{A_1}) u.
\] As we observed below \nn{preslap}, provided we work in an Einstein
scale $\g$, then each factor $(I^{A_i}D_{A_i})$ may be re-expressed in
the form $\si^{-1}(\Delta^{\g}+\lambda_i)$. The value of the constant
$\lambda_i$ ref\/lects the conformal weight of terms to its right. The
scale $\si$, which gives $\g=\si^{-2}\bg$, is parallel for the
Levi-Civita connection $\nabla^\g$, and so the factors of $\si$
cancel, at least after also replacing the density~$u$, as well as the
density-valued operators and curvature $\J$ with their unweighted
equivalents.  We arrive at \nn{powers}. We will term \nn{defP} a GJMS
operator including for the high $k$ in even dimensions where the GJMS
operators were not def\/ined.

Specialising to a Poincar\'e--Einstein manifold we use the observation
in \nn{slap} to suggestively re-express $P_k u$ ($u\in
\mathcal{E}[\frac{k-n}{2}]$) as a composition of scattering Laplacians
\begin{gather}\label{scp}
 P_{k} f = (\Delta^{\g}-s_{k/2}(n-s_{k/2}))\circ \cdots \circ
(\Delta^{\g}-s_{2}(n-s_{2})) \circ (\Delta^{\g}-s_{1}(n-s_{1})) f ,
\end{gather}
where $f=\si^{\frac{n-k}{2}}u$ is the function equivalent to $u$ in
the trivialisation of $\mathcal{E}[\frac{k-n}{2}]$ af\/forded by $\si$.  Since
$u$ has weight $w_0:=(k-n)/2$, and each factor $I^AD_A$ lowers weight
by 1 unit, we have $s_i=n+w_0+1-i=(k+n)/2+1-i$, for $i=1,\dots ,k$.

\subsection{Algebraic decompositions}\label{adec}

\newcommand{\bF}{\mathbb{F}}
\newcommand{\id}{\operatorname{id}}
\newcommand{\cR}{\mathcal{R}}

We digress brief\/ly to recall some rather general considerations from
the work \cite{GoSiDec} with \v Silhan.  Let $\mathcal{V}$ denote a vector
space over a f\/ield $\bF$. Suppose that $\mathcal{P}:\mathcal{V}\to \mathcal{V} $ is a linear
operator that may be expressed as a composition
\[
\mathcal{P}=\mathcal{P}_0\mathcal{P}_1\cdots \mathcal{P}_\ell,
\]
where the linear operators $ \mathcal{P}_i:\mathcal{V}\to \mathcal{V}$,
 $i=0,\dots ,\ell $, are mutually commuting. One might hope that we
 can characterise the range space $\cR(\mathcal{P})$ and null space $\mathcal{N}(\mathcal{P})$ of
 $\mathcal{P}$ in terms of data for the factors $\mathcal{P}_i$. This is straightforward if the
 $\mathcal{P}_i$ are each invertible, but in fact far weaker conditions
 suf\/f\/ice to make signif\/icant progress in this direction.  One
 situation which is particularly useful is as follows.  Suppose that
 we there are linear operators $Q_i:\mathcal{V}\to \mathcal{V}$, $i=0,1,\dots, \ell$,
 that yield a decomposition of the identity,
\begin{gather}\label{iddec}
id_V=Q_0\mathcal{P}^0+\cdots+Q_\ell \mathcal{P}^\ell,
\end{gather}
where $\mathcal{P}^i:=\Pi_{i\neq j=0}^{j=\ell} \mathcal{P}_i, i=0,\dots ,\ell$;
and the $\mathcal{P}_i$s and the $Q_j$s are mutually commuting in that
\[
 \mathcal{P}_iQ_j=Q_j\mathcal{P}_i,\qquad i,j\in\{0,\dots,\ell\}.
\]
This is suf\/f\/icient to give a 1-1 relationship
between solutions $u\in \mathcal{V}$ of the inhomogeneous problem $\mathcal{P} u=f$ and
 solutions $(u_0,\dots
,u_\ell)\in\oplus^{\ell+1}\mathcal{V}$ of the problem
\[
\mathcal{P}_0 u_0=f, \quad \dots , \quad \mathcal{P}_\ell u_\ell =f,
\]
see \cite[Theorem 2.2]{GoSiDec}.
Thus for example the range of $\mathcal{P}$ is exactly the intersection of the
 range spaces for the components $\mathcal{P}_i$.
The map from $u$,  solving $\mathcal{P} u=f$, to solutions of the system is obvious:
\[ u\mapsto (\mathcal{P}^0 u, \dots ,\mathcal{P}^\ell u ) .
\]
One key point is that \nn{iddec} gives an
inverse by
\[
(u_0,\dots, u_\ell)\mapsto \sum_{i=0}^{i=\ell}
Q_i u_i.
\]

Important for us here is that, given the situation above, then for
each $i\in\{0,\dots ,\ell \}$, we have
\[
Q_iP^i:\ \mathcal{N}(\mathcal{P})\to \mathcal{N}(\mathcal{P}_i)
\] and this is a projection. Thus we obtain a direct decomposition of
the null space $\mathcal{N}(\mathcal{P})$. In applying these results in the case that
the $\mathcal{P}_i$ are partial dif\/ferential operators we should expect that
in general the $Q_i$, when they exist, will be pseudo-dif\/ferential
operators. However remarkably there are a large class of situations
where we can solve \nn{iddec} algebraically. For example, for the case
of partial dif\/ferential operators $\mathcal{P}_i$ it can be that the $Q_j$ are
again dif\/ferential and obtained algebraically from the formulae for
the $\mathcal{P}_i$.  The very simplest situation of this is in fact exactly
what we need here and is as follows. This is a special case from
Theorem 1.1 of \cite{GoSiDec}.
\begin{theorem} \label{fundthm}
Let $\mathcal{V}$ be a vector over the field $\mathbb{F}$.
Suppose that $E$ is a linear endomorphism on~$\mathcal{V}$, and $P=P[E]:\mathcal{V}\to \mathcal{V}$ is a linear operator polynomial in
$E$ which factors as
\[ P[E] = (E-\mu_1) \cdots (E-\mu_p), \]
where the scalars $\mu_1,\ldots,\mu_p \in \mathbb{F}$ are mutually distinct.
Then the solution space $\mathcal{V}_P$, for $P$, admits a canonical
and unique direct sum decomposition
\begin{gather}
\mathcal{V}_P=\oplus_{i=0}^\ell \mathcal{V}_{\mu_i},
\end{gather}
where, for each $i$ in the sum, $\mathcal{V}_{\mu_i}$ is the solution
space for $E-\mu_i$. The projection $\Proj_i: \mathcal{V}_P\to
\mathcal{V}_{\mu_i}$ is given by the formula
\[ \Proj_i = Q_i \prod_{i \not= j=1}^{j=p} (E-\mu_j),
   \qquad \text{where} \quad
   Q_i = \prod_{i \not= j=1}^{j=p} \frac{1}{\mu_i - \mu_j}. \]
\end{theorem}

On Einstein manifolds that are not Ricci f\/lat it is easily verif\/ied
that the constants $\la_i$, appearing in the expression \nn{powers}
for $P_k$, satisfy $(\la_i=\la_j)\Rightarrow (i=j)$, $i,j\in
\{1,\dots, k/2\}$. Thus we have exactly the situation of the Theorem
above, and it follows that the solution
space for~$P_k$ decomposes directly. (See \cite{GoSiDec} for further
details and \cite{GoSiEforms} for applications as well as a similar
treatment of operators on dif\/ferential forms.) In particular, from the
Theorem and \nn{scp}, we have the following.
\begin{proposition}\label{PEP}
 On  the
interior $M_+$ of a Poincar\'e--Einstein manifold we have
\begin{gather}\label{Pdecomp}
\mathcal{N}(P_k)=\oplus_{i=1}^{k/2}\mathcal{N}(\Delta^{\g}-s_{i}(n-s_{i})),
\end{gather}
where $s_i= \frac{k+n+1-2i}{2}$, for $i=1,\dots ,k$.
\end{proposition}

\subsection[Dirichlet-Neumann maps from $P_k$]{Dirichlet--Neumann maps from $\boldsymbol{P_k}$}\label{DNP}

We consider the situation f\/irst for the standard conformal
Dirichlet--to--Robin operator. That is, for the source problem we use the
Yamabe Dirichlet problem $(P_2,\delta_0)$, while for the second part of
the construction \nn{comp} we use the conformal Robin operator
$\delta= n^a\nabla^g_a+\frac{n-1}{2}H^g$, from Section~\ref{Bops}.

Recall that the Einstein scale $\si$ is a def\/ining density for the
boundary $\Sigma$, and along $\Sigma$ we have $n^a\nabla^g_a
\si=n^an_a=1$.  It follows easily \cite{BrGonon} from the conformal
transformation of the mean curvature that one can choose the metric
$g$ on $M$ so that $H^g=0$ (so then $\Sigma $ is totally geodesic).
Let us henceforth use $g$ to mean such a metric.  Consider a possible
solution $u$ to
\begin{gather}\label{P2s}
P_2u=0 \qquad \mbox{of the form}\quad  u=U_{\rm o}+\si U_{\rm i},
\end{gather}
where $U_{\rm o}$ and $U_{\rm i}$ are smooth and $\delta U_{\rm o}|_\Sigma=0$.
Note that $\delta \si=1$ along $\Sigma$.
Given
unique solvability of the source problem, such a solution $u$ would reveal a
conformal Dirichlet--to--Neumann map, from Theorem \ref{psidops}, with
$P_{2,{\bf m}_0,0}(f)=U_{\rm i}|_\Sigma$ where $f= U_{\rm o}|_\Sigma$.

Suppose that $\tau$ is the scale determining $g$,
i.e.\ $g=(\tau)^{-2}\bg$. And set $x:=\si/\tau$.
From Proposition  \ref{PEP} it follows that,
in terms of the scale
$\g$, the problem \nn{P2s} is
 equivalent to
\[\big(\Delta^{\g}-s(n-s)\big) u^{\g}=0 \qquad \mbox{with} \quad
u^{\g}= x^{n-s}U^{g}_{\rm o}+ x^{s} U^{g}_{\rm i},
\] where $s=\frac{n+1}{2}$. Here $u^{\g}$ is the function equivalent
 to the density $u$ with respect to the trivialisation of
 $\mathcal{E}[\frac{1-n}{2}]|_{M_+}$ af\/forded by $\si$, that is
 $u^{\g}=\si^{(n-1)/2}u$. On the other hand $U^{g}_{\rm o}$ and
 $U^{g}_{\rm i}$ are the functions equivalent, via the scale $g$ to,
 respectively, $U_{\rm o}$ and $U_{\rm i}$. For example $U^g_{\rm
 o}=\tau^{(n-1)/2}U_{\rm o}$.  Now according to \cite{GG} (using
 \cite{Gu,GrZ}), provided $s(n-s)$ is not an $L^2$ eigenvalue of
 $\Delta^{\g}$, the Dirichlet problem here is uniquely solved by
 solutions of this form and so $U^{g}_{\rm o}|_{\Sigma}\mapsto
 U^{g}_{\rm i}|_{M_+} $ is the scattering map of Graham--Zworski
 \cite{GrZ}. So this is seen to agree with the map $f\mapsto P_{2,{\bf m}_0,0}(f) $ of Theorem \ref{psidops}. In fact in \cite{GG} they make
 exactly this point: that for $s=(n+1)/2$ the scattering map agrees
 with a Dirichlet--to--Neumann map.

This situation for the higher order GJMS operators is partly similar
as follows. Consider Dirichlet--Neumann operators constructed as in~\cite{BrGonon}, i.e.\ as in Theorem~\ref{psidops}, except using a GJMS
operator $P_k$ as the interior operator (rather than~$\Box_k$). There
is the question of whether there are suitable boundary operators for
$P_k$, to replace the $\delta'_{\bf m}$, of Proposition~\ref{selfad}
and Theorem~\ref{psidops}. Rather than confront this possibly
dif\/f\/icult issue at this point, let us simply assume that there are
such operators and denote these also by $\delta'_{\bf m}$. That is we
will assume that we have all the conditions required for Theorem~\ref{psidops}, with now $P_k$ everywhere replacing $\Box_k$ in that
Theorem. This assumption is not totally outrageous: For $P_2$, as
above, we have $P_2=\Box_2$ and similarly for $P_4$, provided $d\neq
4$, this is satisf\/ied as $P_4$ is a non-zero multiple of
$\Box_4$. Similarly on conformally f\/lat manifolds also this is
satisf\/ied for all $k$ in the range covered in Theorem \ref{psidops},
as again we have that~$P_k$ and~$\Box_k$ agree. (See \cite{GoPetCMP}
for these last facts.)

It will shortly be clear that for comparison with \cite{GrZ}, the
source problem $(P_k,\delta'_{{\bf m}_0})$ (termed the {\em generalised
Dirichlet problem} in \cite{BrGonon}) is relevant. So suppose that this
uniquely solvable. (Note from \nn{Pdecomp} this requires that
$s_i(n-s_i)$ is not an $L^2$ eigenvalue of $\Delta^{\g}$, for $s_i$ as
in Proposition~\ref{PEP}.)  Then with
Dirichlet boundary data as in \nn{SetBdryCond} we get a
Dirichlet--to--Neumann map akin to~$P_{k,{\bf m},m_j}$.  According to
Proposition \ref{PEP}, any solution $u$ is a direct sum
$u=u_1+u_2+\cdots +u_{k/2}$, where $(\Delta^{\g}-s_{\ell}(n-s_{\ell}))
u^{\g}_\ell=0$.  However from this perspective it is not immediately
clear how, in general, to relate the boundary data from the $P_k$
problem to boundary data for the solution~$u^{\g}_{\ell}$ of the
scattering Laplacian. There are several dif\/f\/iculties here. For
example, according to Theorem \ref{fundthm}, the projections $u\mapsto
u_{\ell}$ are administered by dif\/ferential operators which, as given,
do not make sense on $\Sigma$.

In the other direction,  suppose we have a solution $u^{\g}$ to
$(\Delta^{\g}-s_{j}(n-s_{j}))$ ($j\in \{ 1,\dots ,k/2\}$) of the form
\[
u^{\g}= x^{n-s_j} U_{\rm o}^{g}+x^{s_j} U_{\rm i}^g
\]
with $s_j= \frac{k+n-1-2m_j}{2}$ and $m_j\in {\rm m}_0$. (Since $2s_\ell-n$ is
odd and the Poincar\'e--Einstein metric is suitably even, in the sense described
pp.~108--109 of \cite{GrZ}, it follows that the coef\/f\/icient of $G$ in
\cite[Proposition 3.5]{GrZ} vanishes. See the proof of Proposition 4.2
in \cite{GrZ}, or Lemma 4.1 in \cite{Gu}. Thus we expect solutions of
the above form provided that $s_\ell(n-s_\ell)$ is not an $L^2$
eigenvalue of~$\Delta^{\g}$.)  Then using that $x=\si/\tau$ we may
re-express the solution in terms of densities:
\[
u=\si^{m_j}U_{\rm o} + \si^{k-1-m_j}U_{\rm i},
\] with (using \nn{scp}) $u\in \mathcal{E}[\frac{k-d}{2}]$ solving $P_k u=0$
and $U_{\rm o}\in\mathcal{E}[\frac{k-n-2m_j-1}{2}] $ and $U_{\rm i}\in
\mathcal{E}[\frac{-k-n+2m_j+1}{2}] $.  Thus $U_{\rm o}^{g}|_\Sigma$ is the
function equivalent to the conformal density $ U_{\rm o}|_\Sigma=
c\cdot\delta'_{m_j} u|_\Sigma$, for some non-zero constant $c$.  We
also clearly have that $\delta'_{m_i}u|_\Sigma=0$ for integers $i$,
$0\leq i <j$, and $\delta'_{k-1-m_j} (\si^{k-1-m_j}U_{\rm i})|_\Sigma$
is a nonzero constant times $U_{\rm i}|_\Sigma$.  For example if
$j=k/2-1$, then $u$ is the unique solution to the $(P_k,\delta'_{{\bf
m}_0})$ problem with $\delta'_{m_j} u|_\Sigma$ prescribed to agree
with $c^{-1}\cdot  U_{\rm o}|_\Sigma$.  However we cannot in general
say more to compare the scattering map with the Dirichlet--to--Neumann
map without considerable explicit information about the asymptotics of $u$.

It is also clear that this comparison will be sensitive to the details
of the boundary operators~$\delta'_{r}$ used. The observation at the
beginning of Section \ref{scat}, that for most weights $I^AD_A$
recovers the conformal Robin operator $\delta$, suggests the idea that
there are likely to be higher order analogues of $\delta$ (i.e.\
variants of the $\delta'_r$) that are well adapted to the GJMS
operators on Poincar\'e--Einstein manifolds. Since such operators could
signif\/icantly simplify the conformal Dirichlet--to--Neumann construction
and its relationship to the the scattering map it seems that investigating
this possibility should be the next step in the programme.

\section{New directions: translating}\label{trans}

Recall from Section \ref{hypersec} there is a tractor twisting of the
Yamabe operator $\Box:\mathcal{T}^*[1-d/2]\to \mathcal{T}^*[-1-d/2] $. Let us assume
that on some Poincar\'e--Einstein manifold  $(M^d,[g],I)$ with boundary~$\Sigma$ the Dirichlet problem for $\Box: \mathcal{T}^*[1-d/2]\to
\mathcal{T}^*[-1-d/2] $ is uniquely solvable. For example this is the case on
the homogeneous model $(C,[g])$ from Section \ref{model}; since on $C$
the tractor bundle is trivialised by parallel sections this follows
easily from the unique solvability of the density problem, as
discussed in e.g.  \cite{BrGonon}. Since the conformal Robin operator
$\delta$ is also strongly invariant it follows that we may construct a
tractor twisted conformal Dirichlet--to--Neumann map,
\[
P^{\mathcal{T}}_{2,{\bf m}_0,0}:\mathcal{T}^*_\Sigma\left[\frac{1-n}{2}\right]\to
\mathcal{T}^*_\Sigma\left[\frac{-1-n}{2}\right].
\]
Here we are using some key facts. Firstly $\mathcal{T}^*_\Sigma$ may be
identif\/ied with the subbundle of the restriction to $\Sigma$ of some
ambient tractor bundle $\mathcal{T}^*$, and that this subbundle is
characterised by being the part of $\mathcal{T}^*|_\Sigma$ annihilated by any
contraction with the normal tractor $N^A$. Next since $I^A$ is
parallel and recovers $N^A$ along $\Sigma$, it follows from the unique
solvability that any solution $u$ to the Dirichlet problem has any contraction with $I^A$ vanishing everywhere. (Note that any such contraction itself solves
a Yamabe Dirichlet problem, but with Dirichlet data the zero section.)  From
these observations it follows that the map $P^{\mathcal{T}}_{2,{\bf m}_0,0}$
takes values in the bundle~$\mathcal{T}^*_\Sigma[\frac{-1-n}{2}]$.

Next suppose $\mathcal{U}^*$ is an irreducible (conformally weighted) tensor
bundle on $\Sigma$ and there is a~conformal dif\/ferential operator $S:
\mathcal{U} \to \mathcal{T}^*_\Sigma[\frac{1-n}{2}]$ with conformal formal adjoint (as discussed in
e.g.~\cite{BrGodeRham}) $S^*: \mathcal{T}^*_\Sigma\to \mathcal{U}_*$. Here $\mathcal{U}_*$ is
$\mathcal{U}\otimes \mathcal{E}[w]$ where the weight $w$ is such that the natural
pairing of a section of $\mathcal{U}^*$ with a section of $\mathcal{U}_*$, via the
conformal metric, yields a density of weight $-n$. Then we may form
the  composition
\[
\mathcal{U}^*\stackrel{S}{\to}\mathcal{T}^*_\Sigma\left[\frac{1-n}{2}
\right]\stackrel{P^{\mathcal{T}}_{2,{\bf m}_0,0} }{\longrightarrow}
\mathcal{T}^*_\Sigma\left[\frac{-1-n}{2}\right] \stackrel{S^*}{\to} \mathcal{U}_*;
\]
by construction this composition $\mathcal{P}:=S^*\circ P^{\mathcal{T}}_{2,{\bf
    m}_0,0}\circ S $ is conformally invariant. It is easily verif\/ied
that $P^{\mathcal{T}}_{2,{\bf m}_0,0} $ is self-adjoint (by an adaption of the
argument for $P_{2,{\bf m}_0,0}$), so $ \mathcal{P}$ is formally self-adjoint
by construction.

The natural candidates for the operators $S: \mathcal{U}^* \to
\mathcal{T}^*_\Sigma[\frac{1-n}{2}]$ are the so-called dif\/ferential splitting
operators; $S$ is of this form if there is a bundle map $T$ from a
subbundle of $\mathcal{T}^*_\Sigma[\frac{1-n}{2}]$ to $\mathcal{U}^*$ satisfying $
T\circ S=id_{\mathcal{U}^*}$. There is a rich and well developed
theory for the construction of such splitting operators see for
example \cite{Esrni,GoSilKil}; for conformal geometry a general and
practical construction is developed in \cite{Sithesis}, while for an
elegant recent advance which applies to all parabolic geometries see
\cite{CapSouCas}.

Let us illustrate with a simple example. Let $d=4$, take
$\mathcal{T}^*_\Sigma$ to be simply the standard tractor bundle $\mathcal{T}_\Sigma$
on $\Sigma$ and for $\mathcal{U}^*$ take the cotangent bundle $\mathcal{E}^1_\Sigma$.
There is a splitting operator (see e.g.~\cite{Esrni})
\[
E: \ \mathcal{E}^1_\Sigma \to \mathcal{T}_\Sigma[-1] \qquad \mbox{by} \quad \phi_b \mapsto
\left(\begin{array}{c}0\\
\phi_a\\
-\nabla^c\phi_c \end{array}\right).
\]
Thus on the conformal 3-manifold $\Sigma$ we obtain
$\mathcal{P}:\mathcal{E}^1_\Sigma \to \mathcal{E}^1_\Sigma[-1]$ by the
composition $\mathcal{P}= E^* \circ P^{\mathcal{T}}_{2,{\bf m}_0,0}
\circ E$.  To see this is non-trivial we argue as follows.  On the
homogeneous model the Dirichlet and conformal Neumann problems for
$\Box: \mathcal{T}[1-d/2]\to \mathcal{T}[-1-d/2] $ are equivalent to
trivial twistings of the Dirichlet and conformal  Neumann problems for~$P_2$. Thus they are
both properly elliptic normal boundary problem satisfying the
Lopatinski--Shapiro conditions and so each has f\/inite dimensional
kernel. See \cite[Proposition 6.4]{BrGonon} for a summary of the
relevant facts from \cite{grubb,hormander}. It follows immediately
that the composition $P^{\mathcal{T}}_{2,{\bf m}_0,0}\circ E$ is
non-trivial.  Then using that $E$ is $G$-invariant splitting operator
and considering the possible $G={\rm SO}(n+1,1)$ intertwinors between
$\mathcal{E}^1$ and other irreducibles \cite{specgen} it follows
easily that $E^* P^{\mathcal{T}}_{2,{\bf m}_0,0}\circ E$ is
non-trivial. From this point the non-triviality of this operator in
general can be established from the pseudo-dif\/ferential nature of the
operator (it is a composition of dif\/ferential and pseudo-dif\/ferential
operators) and symbol analysis.

It seems likely that a large class of integral order pseudo-dif\/ferential
operators will arise from the construction sketched above.  For
example a non-linear conformal tensorial Dirichlet--to--Neumann map was
announced in \cite{GrDN}; in view of the uniqueness of intertwinors, its
linearisation should be  recoverable using these ideas.

There is scope to develop a
similar translation of the scattering construction to yield tensorial
Dirichlet--to--Neumann maps in the case of non-half integral weights. Here a
key point is that, on the one hand, the operator $I^AD_A$ extends the
scattering Laplacian to the boundary of the Poincar\'e--Einstein
manifold (where it degenerates to a constant times $\delta$), while on
the other it is a strongly invariant operator.  Given a tractor bundle
$\mathcal{T}^*[w]$ (of some possibly complex weight $w$) an idea for extending
data of\/f the boundary is to use sections $u\in\mathcal{T}^*[w]$ satisfying
$I^AD_A u=0$ and satisfying the compatible property that $u$ is
annihilated by any contraction with~$I^A$.  In particular we may seek
solutions of the form $u =\si^z U_{\rm o} + \si^{2w+n-z}U_{\rm i}$
where the tractor bundle sections $U_{\rm o}$ and $U_{\rm i}$ are
smooth, have appropriate weights, and are annihilated by any
contraction with $I^A$ (and as usual $\si$ is the Einstein scale
$\si=h(X,I)$).  The situation is most clear on the homogeneous model
$(C,[g])$ from Section \ref{model}.  Once again using that, in this
case, the tractor bundles are trivialised by a parallel frame it
follows from the density case that the required Poisson operators
exist for a~set of weights dense in $\mathbb{C}$.  In any case, given
a~scattering map $U_{\rm o}|_\Sigma\to U_{\rm i}|_\Sigma$ one may
translate to maps between weighted tensor bundles by composing fore
and aft with dif\/ferential splitting operators as for the construction
above.

\subsection*{Acknowledgements}

ARG gratefully acknowledges support from the Royal Society of New
Zealand via Marsden Grant no.\ 06-UOA-029.  It is a pleasure to thank
Andreas \v Cap, Robin Graham, Colin Guillarmou and Andrew Hassell for
helpful discussions.

\pdfbookmark[1]{References}{ref}
\LastPageEnding

\end{document}